%% file: linearizing.tex
\documentclass{amsart}
\usepackage{etoolbox}
\makeatletter
\patchcmd{\@sect}
  {\interlinepenalty\@M #8\par}
  {\interlinepenalty\@M #8\unskip\@addpunct.\par}
  {}{}
\makeatother

\usepackage[alphabetic,initials]{amsrefs}
\usepackage{enumerate}
\usepackage{amssymb}
\usepackage{stmaryrd}
\usepackage{float}
\theoremstyle{plain}
\newtheorem{theorem}{Theorem}[section]

\newtheorem{lemma}[theorem]{Lemma}

\theoremstyle{definition}
\newtheorem{definition}[theorem]{Definition}

\newtheorem{remark}[theorem]{Remark}

\newtheorem{example}[theorem]{Example}
\makeatletter
\AtBeginEnvironment{example}{%
  \pushQED{\qed}%
}
\AtEndEnvironment{example}{%
  \popQED
}
\makeatother
\newtheorem{question}[theorem]{Question}

\usepackage[bookmarksnumbered=true]{hyperref}
\usepackage{xcolor}
\hypersetup{
  colorlinks = true,
  linkcolor = blue,
  anchorcolor = blue,
  citecolor = teal,
  filecolor = blue,
  urlcolor = blue
}
\usepackage{tikz}
\usepackage{tikz-cd}
\usetikzlibrary{calc}
\usetikzlibrary{intersections}

\DeclareMathOperator{\rank}{rank}
\DeclareMathOperator{\Char}{char}
\renewcommand{\d}{\mathrm{d}}
\usepackage{graphicx} 
\usepackage{verbatim}
\usepackage{blkarray}
\title{Linearizing algebraic matroids}

\input{authors}

\begin{document}

\begin{abstract}
  Although algebraic
  matroids were discovered in the 1930s, interest in them was largely
  dormant until their recent use in applications
  of algebraic geometry.
  Because nonlinear algebra is computationally challenging, it is
  easier to work with an
  isomorphic linear matroid
  if one exists.  We  describe an explicit construction
  that produces a linear
  representation over an algebraically closed field of characteristic
  zero starting with the data used in applications.
  We will also discuss classical examples of algebraic matroids
  in the modern language of polynomial ideals, illustrating how the
  existence of an
  isomorphic linear matroid depends on properties of the field.

\end{abstract}

\maketitle

\input{newIntro}

\input{proofs}

\input{parameterized_varieties}
\input{history}

\section{Acknowledgments and declarations}
\subsection{Acknowledgments}
\input{acknowledgements}
\subsection{Funding}
\input{funding}
\subsection{Conflicts of interest}
\input{declarations}
\subsection{Contributions}
\input{contributions}

\input{biblio}
\end{document}

%% file: authors.tex
\author{Zvi Rosen}
\address{Department of Mathematics and Statistics, Florida Atlantic
University, Boca Raton, FL USA}
\email{rosenz@fau.edu}

\author{Jessica Sidman}
\address{Department of Mathematics and Statistics, Amherst College,
Amherst, MA USA}
\email{jsidman@amherst.edu}

\author{Louis Theran}
\address{School of Mathematics and Statistics, University of St
Andrews, St Andrews, Scotland}
\email{lst6@st-andrews.ac.uk}

%% file: newIntro.tex
\section{Introduction}
Recent applications of algebraic geometry to areas including
algebraic statistics \cites{GS,HS}, chemical reaction networks \cites{FS,GHRS},
rigidity theory \cite{GGJ},
and coding theory \cite{BD},
have spurred interest in a reformulation of the definition of an
algebraic matroid in terms of polynomials.  In this definition,
the dependent sets of coordinates are those which are the support
of some polynomial in a prime ideal $P \subseteq k[x_1,\ldots, x_n]$ for some field $k$.
However, algebraic matroids
are difficult to work with computationally, since testing for independence
requires elimination theory.  It is more convenient to
work with an isomorphic {\em linear} matroid, in which  testing independence
can be done with linear algebra.  A widely-cited result of Ingleton \cite{I71}
says that passing to a linear matroid is possible if $k$ has
characteristic zero and one is willing to work in an extension of $k$.
However, a folklore extension of Ingleton's result says that, when
$k$ is algebraically closed, there exists a
$k$-linear representation.

We present both versions of Ingleton's result, using
elementary algebraic geometry and linear algebra.  The proof shows
that Ingleton's extension of $k$ may be taken to be
the field of fractions of $k[x_1, \ldots, x_n]/P$, denoted $\operatorname{Frac}(k[x_1, \ldots, x_n]/P)$.  
It is well known to algebraic geometers that, when $k$ is 
algebraically closed, one can pass from the $\operatorname{Frac}(k[x_1, \ldots, x_n]/P)$-linear representation 
to $k$-linear one, but this fact is usually presented in scheme-theoretical language.  
We give a more accessible exposition.
We also discuss why the hypotheses that $k$
has characteristic zero and is algebraically closed cannot always be
dropped, as well
as a common situation where the latter hypothesis can be relaxed.

To illustrate the flavor of how algebraic matroids are used in
applications, consider the
following example of a matrix completion problem.
\begin{example}\label{ex: prob}
Suppose we have a $2 \times 2$ matrix $X$ with real entries, and we know that $\rank X= 1.$ If \[
   X = \begin{pmatrix}
      x_{00} & x_{10} \\
      x_{01} & x_{11}
    \end{pmatrix},\]
    the entries of $X$ satisfy the polynomial relation $x_{00}x_{11} -
  x_{10}x_{01} = 0$.  This equation shows us that we can choose any three of the entries independently, but we expect the fourth entry to be determined by the others.  This dependence is exhibited for any sufficiently general choice of three real numbers (for example, if they are all nonzero).
    
  In fact, this is an example of an algebraic condition satisfied by the joint distribution of two independent 0/1 random variables $Y_1$ and $Y_2$.  To hint at the connection, note that $\rank X = 1$ if and only if it factors as 
  \[
    \begin{pmatrix}
      x_{00} & x_{10} \\
      x_{01} & x_{11}
    \end{pmatrix}
    =
    \begin{pmatrix}
      p_0 q_0 & p_1 q_0 \\
      p_0 q_1 & p_1 q_1
    \end{pmatrix}
    =
    \begin{pmatrix}
      q_0 \\ q_1
    \end{pmatrix}
    \begin{pmatrix}
      p_0 & p_1
    \end{pmatrix}
  \]
  for some $p_1,p_2, q_1, q_2\in \mathbb{R}.$
  In algebraic statistics, this is the key condition characterizing 
  the independence model. With additional restrictions, 
  we can interpret the probability that $Y_1 = i$ as $p_i$ and the 
  probability that $Y_2 = i$ as $q_i,$ and then $x_{ij}$ represents 
  the probability that $Y_1 = i$ and $Y_2 = j.$  See 
  \cite[Definition 1.1.3 and Example 3.1.6]{owl.pdf} for more details. 
\end{example}
The preceding example has all the key ingredients motivating applied
interest in algebraic matroids.  The
coordinates have a meaningful interpretation (here, as
probabilities $x_{ij}$), and these coordinates
must satisfy some nonlinear polynomial relations imposed by the
application (here, the determinant certifying that $X_1$ and $X_0$
are independent).
It is then natural to ask the purely combinatorial question: how do
we capture which subsets of
coordinates are dependent and which are independent?  This is precisely the data captured by a matroid.

\begin{definition}\label{def: matroid}
  Let $E$ be a finite set and $\mathcal{I}$ be a collection of
  subsets of $E$ satisfiying the following properties:
  \begin{enumerate}
    \item $\emptyset \in \mathcal{I}$;
    \item if $I \in \mathcal{I}$, then every subset of $I$ is in $\mathcal{I}$; and
    \item if $I, I' \in \mathcal{I}$ and $|I| < |I'|$ there exists $x
      \in I'\backslash I$ such that $I \cup \{x\}$ is in $\mathcal{I}.$
  \end{enumerate}
  We define $M = (E, \mathcal{I})$ to be a {\em matroid} with ground
  set $E$ and {\em independent sets} in $\mathcal{I},$ and any
  subset of $E$ not in $\mathcal{I}$ is called {\em dependent}.
\end{definition}
\begin{example}\label{ex: small matroid}
Let $E = \{1,2,3\}$ and $\mathcal{I} = \{\emptyset, \{1\},\{2\},\{3\},\{1,3\},\{2,3\}\}$.
One can check that $(E, \mathcal{I})$ is a matroid.
\end{example}
We will be
interested in two ways that these abstract matroids can be
represented as familiar objects and how these objects are related to
one another.
\begin{definition}\label{def: algebraic matroid}
  Let $M = (E, \mathcal{I})$ be a matroid with a ground set of size
  $n$ and $k$ be a field.
  \begin{itemize}
    \item We say that $M$ is {\em $k$-representable} if
      there is a $k$-vector space $V$ and a map $\varphi : E\to V$ so
      that $I\in \mathcal{I}$ if and only if $\varphi(I)$ is
      linearly independent.

    \item   We say that $M$ is {\em $k$-algebraic} if there is
      a prime ideal $P\subseteq k[x_1, \ldots, x_n]$
      and a map $\varphi : E\to \{x_1, \ldots, x_n\}$ so that $I\in
      \mathcal{I}$ if and only if $P\cap k[\varphi(I)] = \langle 0\rangle$; i.e., 
      that $\varphi(I)$ is algebraically independent mod $P$.
  \end{itemize}
\end{definition}
Informally, $M$ is $k$-representable if its independent sets correspond
to the linearly independent subsets of a
vector configuration in a $k$-vector space.
That this gives a matroid follows immediately from basis extension
theorems in linear algebra.  Indeed, the word
``matroid'' is meant to indicate that it generalizes the independence
structure determined by the columns of a matrix.  Finally, verifying linear independence of vectors can be 
performed via Gaussian elimination, which is efficient in terms of the number of additions and 
multiplications in $k$.

\begin{example}\label{ex: small matroid rep}

The matroid from Example \ref{ex: small matroid} is $\mathbb{R}$-representable.  We can take 
$\varphi$ to map $i$ to the $i$th column of the matrix 
\[
\begin{pmatrix}
    1 & -1 & 0 \\
    0 & 0  & 1
\end{pmatrix}. \qedhere
\]

\end{example}

Algebraic matroids are
more complicated.  Let $P\subseteq k[x_1, \ldots, x_n]$ be a prime ideal and $M(P)$ denote the induced algebraic matroid.  
From the definition, a set of variables $I$ is independent in $M(P)$ if
the ideal $P$ does not contain a polynomial supported
on $I$, and verifying this requires {\em polynomial} elimination, which is computationally
difficult.  
\begin{example}\label{ex: small matroid algebraic}
The matroid from Example \ref{ex: small matroid} is also $\mathbb{R}$-algebraic, as is 
any $\mathbb{R}$-representable matroid.  For example, we can set 
$P = \langle x_1 + x_2\rangle\subseteq \mathbb{R}[x_1,x_2,x_3]$.  Then 
the map $\varphi(i) = x_i$ certifies that the matroid is algebraic.
\end{example}


Our main topic is how to move from an algebraic matroid to an 
equivalent linear one.  Here we show how to do it for the 
algebraic matroid from Example \ref{ex: prob}.
\begin{example}\label{ex: diff}
  Let $P = \langle x_{00}x_{11} -
  x_{10}x_{01} \rangle \subseteq \mathbb{R}[x_{00}, x_{01}, x_{10},
  x_{11}]$ be the prime ideal generated by the equation from Example
  \ref{ex: prob} and $M(P)$ be the
  associated algebraic matroid.  Because $P$ has a single generator
  supported on all four variables, the independent sets of $M(P)$ are all the
  subsets of $\{x_{00}, x_{01}, x_{10}, x_{11}\}$ of size at most
  three.  Hence, any sufficiently general choice of $\varphi(x_{ij})
  \in \mathbb{R}^3$
  demonstrates that $M(P)$ is $\mathbb{R}$-representable.

  However, let us proceed in a more principled way, using the
  presentation of the
  ideal $P$ and differentials, which behave in a way that is  familiar from calculus.  (We will give a formal definition in Section~\ref{sec: diff}.)  Let 
  $\mathbb{R}(P)$ be the field of fractions 
  $\operatorname{Frac}(\mathbb{R}[x_{ij}]/P)$ of the 
  coordinate ring $\mathbb{R}[x_{ij}]/P$ and $\{\d x_{00}, \d x_{01}, \d x_{10}, \d x_{11}\}$ be a basis 
  for an $\mathbb{R}(P)$-vector space $V$.    This 
  choice of field is the right one, because the coefficients 
  of ${\d f}$ are polynomials in $\mathbb{R}[x_{ij}]$, and we 
  want to treat elements of $P$ as zero.
  
  Now define $U \subseteq V$ to be the subspace spanned by
  \[
    \d( x_{00}x_{11} - x_{10}x_{01}) = x_{11}\d x_{00} - x_{10}\d
    x_{01} - x_{01}\d x_{10} + x_{00}\d x_{11}.
  \]
  The linear matroid corresponding to $M(P)$ is represented by the images of   $\d x_{ij}$ in  $V/U$. %
 
  The matrix of the inclusion $U\hookrightarrow
  V$ in the basis $\d x_{ij}$ is the transpose of
  \[
    J =
    \begin{pmatrix}
      x_{11} &-x_{10} & - x_{01} & x_{00}
    \end{pmatrix}.
  \]
  Therefore, by the first isomorphism theorem, the columns 
  of any matrix
  $A$ with nullspace equal to the column
  space of $J^T$ generate a vector space that
  is linearly isomorphic to $V/U$. An example of
  such a matrix is
  \[
    A =
    \begin{pmatrix}
      1 & x_{11}/x_{10} & 0 & 0 \\
      1 & 0 & x_{11}/x_{10} & 0 \\
      1 & 0 & 0 & -x_{11}/x_{00} \\
    \end{pmatrix}.
  \]
  Under the
  isomorphism, the columns of $A$
  correspond to the images of $\d x_{ij}$ under the quotient map, and
  we get an  $\mathbb{R}(P)$-representation of $M(P)$ 
  from the columns of $A$. 
  We call this construction the ``Jacobian step''.

  Now we find an $\mathbb{R}$-representation for $M(P)$ as follows.
  We specialize the
  $x_{ij}$ to
  \[
    x_{00} = \frac{1}{6},\, x_{01} = \frac{1}{6}, x_{10} =
    \frac{1}{3}, x_{11} = \frac{1}{3},
  \]
  which specializes $A$ to the matrix
  \[
    A_{(1/6,1/6,1/3,1/3)} =
    \begin{pmatrix}
      1 & 2 & 0 & 0 \\
      1 & 0 & 1 & 0 \\
      1 & 0 & 0 & -1 \\
    \end{pmatrix}.
  \]
  The columns of this last matrix give the same linear matroid as the
  columns of $A$, so we have found an $\mathbb{R}$-representation.
  We call this move the ``specialization step''.
\end{example}

In Example~\ref{ex: diff}, we constructed an $\mathbb{R}$-linear
matroid from an $\mathbb{R}$-algebraic
one.
This leads us to ask:
\begin{question}\label{q: alg implies lin}
  Given a prime ideal $P = \langle f_1, \ldots, f_m\rangle \subseteq
  k[x_1, \ldots, x_n]$,
  can we find a $k$-linear representation of $M(P)$?
\end{question}
Even more ambitiously, we ask whether the same construction 
always works.  The matrix $J$ is the
Jacobian of the polynomial generating $P$, so we could perform the same
computations when there are more generators.  There were two main steps:
the Jacobian step, where we construct a 
 matrix $A$ whose columns give a
$\operatorname{Frac}(k[x_1, \ldots, x_n]/P)$-representation of $M(P)$;
and the specialization step, in which we find a 
point $z\in V(P)$ so that specializing $A$ to $A_z$ gives an isomorphic
$k$-representation.  Thinking about the general
case, each step could go wrong.   If $k$ has  characteristic $p$,
then $\d( x^p) = 0$ even though $x^p$ is nonzero.  This opens the
possibility that $\d f_i$ and $f_i$ have different supports, and
then it is unclear that $A$ gives a representation of $M(P)$.  On
the other hand, if $k$ is not algebraically closed, $V(P)$ may not
have any smooth $k$-points, so our specialization step might not
even get started.
These issues were known in the 1970's and 1980's to 
A. Ingleton and  B. Lindström
who studied the closely related question:
\begin{question}\label{q: alg implies lin b}
  Given a matroid $M$ known to be algebraic, determine if $M$ is 
  linearly representable.
\end{question}
The difference is that Question \ref{q: alg implies lin b} asks only
for existence, and
that the application-motivated Question \ref{q: alg implies lin} starts with a
specific algebraic representation and asks for a linear
representation over the same field.  Our goal in this paper is to
give an account of
Ingleton's and Lindström's work in modern language, and to describe a sufficient
condition for Question \ref{q: alg implies lin} to have a positive solution.

A theme going back nearly to the beginning of matroid theory is the
general project of classifying matroids by their $k$-representability
over different fields $k$, or the lack thereof.  Question \ref{q: alg implies lin b}
is a branch off this line of research, which, like the trunk, was
profoundly impacted by the discovery of the Mnëv--Sturmfels universality
phenomenon in the mid-1980s.  In particular, it seems that obtaining 
an answer to Question \ref{q: alg implies lin b} will be very difficult.
However, the work of Ingleton and Lindström made some fundamental
advances.

Lindström provides negative results in \cites{L83,L86}.
For every finite field $k$, he gives an example of a
matroid which is $k$-algebraic, but not $K$-representable for
{\em any} field $K$.  We will discuss Lindström's
example, the non-Pappus matroid,
in further detail in Section~\ref{sec: lindstrom}.
Ingleton's work, which is the first part of the next theorem, provides
a positive answer to Question \ref{q: alg implies lin b} when $k$ has
characteristic zero.  The second part of the theorem is an answer to
Question \ref{q: alg
implies lin}.
\begin{theorem}\label{thm: K/k-linear}
  Let $k$ be a field of characteristic zero and $M$ be a $k$-algebraic matroid.
  \begin{enumerate}
    \item There is a finitely generated
      extension $K/k$ such that $M$ is $K$-representable  \cite{I71};
    \item  If $k$ is, in addition, algebraically closed, then $M$ is
      $k$-representable.
  \end{enumerate}
\end{theorem}
We note that part (2) of Theorem \ref{thm: K/k-linear} is the 
one which practitioners usually want to invoke, because  
computations in the field extension $K/k$ appearing in part (1) 
require a Gröbner basis.

In Section~\ref{sec: best case} we prove Theorem \ref{thm: K/k-linear} in the
language of ideals and varieties, giving explicit polynomial conditions to check in order to construct the 
$k$-representation in (2) (Theorem~\ref{thm: specialization}).  In
Section~\ref{sec: char zero}, we give an example with $k = \mathbb{Q}$
that shows the hypothesis on $k$ beyond characteristic zero is necessary
to make part (2) true.%
\footnote{In \cite{RST}, Ingleton's \cite{I71} Theorem \ref{thm: K/k-linear}.
(1) is incorrectly stated.} %
 We also discuss why the characteristic zero hypothesis is not needed in the case of parameterized varieties (Theorem~\ref{thm: param}).  Finally, in Section~\ref{sec: lindstrom} we describe Lindström's examples in the
updated language and sketch more recent work on Frobenius flocks, which provide a structured way to associate linear matroids to an algebraic matroid in characteristic $p$.

%% file: proofs.tex
\section{The Best Case: Algebraically Closed and Characteristic Zero}\label{sec: best case}
To prove Theorem \ref{thm: K/k-linear}, we generalize the
reasoning in Example \ref{ex: diff}.  We begin by introducing some language for discussing algebraic geometry and differentials in Sections~\ref{sec: alg} and \ref{sec: diff}.  The main new element, which we discuss in Section~\ref{sec: special}, is that,
to prove part (2), we will need to use the additional hypothesis that
$k$ is algebraically closed to find a point on $V(P)$ to specialize
the linear representation provided by part (1).  With this setup the proof of Theorem~\ref{thm: K/k-linear} follows easily in Section~\ref{sec: proof}

\subsection{Algebraic varieties}\label{sec: alg}
To describe this construction it is helpful to use language from
algebraic geometry.  Let $S = k[x_1, \ldots, x_n]$.  For any
ideal $I\subseteq S$, the \emph{algebraic set} $V(I) = \{ z\in k^n : f(z) =
0\, \ \forall f\in S\}$
is the common vanishing locus of the polynomials in $I$.
The family
of algebraic subsets of $k^n$ is closed under finite unions and
arbitrary intersections.
If $P\subseteq S$ is a prime ideal, we call $V(P)$ an {\em irreducible variety}.
The next lemma, which we state without proof, is a consequence of the Nullstellensatz and describes two 
important properties of irreducible  varieties.
\begin{lemma}\label{lem: noth}
  Let $k$ be an algebraically closed field and $P\subseteq k[x_1,
  \ldots, x_n]$ a
  prime ideal.  Then the irreducible variety $X = V(P)$ is
  \begin{enumerate}[(a)]
    \item nonempty;
    \item not the union of a finite number of proper algebraic subsets of $X$.
  \end{enumerate}
\end{lemma}
Lemma \ref{lem: noth} fails if $k$ is not algebraically closed.  For example 
if $P = \langle x^2 + 1\rangle\subseteq \mathbb{R}[x]$, then $V(P)$
is empty.

\subsection{Differentials}\label{sec: diff}
We also need a result about differentials.  We recall that
if $f\in S$ is a monomial not supported on $x_i$, then
$\frac{\partial}{\partial x_i} f = 0$;
otherwise, we can write $f = x_i^m g$, where $g$ is a monomial
not supported on $x_i$ and then $\frac{\partial}{\partial x_i} f =
mx_i^{m-1}g$.  The
operator $\partial/\partial x_i : S\to S$ is then defined by linear extension.
Now we let $P\subseteq S$ be a prime ideal and define $k(P)$ to be
the field of fractions of $S/P$; for $f\in S$, denote by
$[f]$ its coset in $S/P$.  If $\alpha\in k(P)$, it is of the
form $[p]/[q]$, where $p, q\in S$ and $[q]$ is nonzero.
Now we define $V$ to be  the
$k(P)$-vector space generated by $\{\d x_i\}$. The
differential $\d : S\to V$ is
then defined by
\[
  \d f = \left[\frac{\partial}{\partial x_1}(f)\right] \d x_1 + \cdots +
  \left[\frac{\partial}{\partial x_n}(f)\right]\d x_n.
\]
Since the differential satisfies
the product rule, if $f\in P$ and $q\in S$, we have
\begin{equation}\label{eq: prodrule}
      \d (qf) = [q]\d f + [f]\d q = [q]\d f,
\end{equation}
because $[f]$ is zero in $k(P)$.  Here is a simple observation, that is
true for any $k$.
\begin{lemma}\label{lem: diff spans}
  Let $P = \langle f_1, \ldots, f_m\rangle \subseteq S$
  a prime ideal.  Then the image $\d P$  of $P$ under the differential
  is contained in the $k(P)$-linear span $U\subseteq V$ of $\{\d f_1,
  \ldots, \d f_m\}$.
\end{lemma}
\begin{proof}
  Let $g\in P$.  Then there are polynomials $q_1, \ldots, q_m$ such that
  $g = q_1 f_1 + \cdots + q_m f_m$.  Using \eqref{eq: prodrule}, we compute
  \[
    \d g = ([q_1] \d f_1 + [f_1] \d q_1) + \cdots +  ([q_m]\d f_m +
    [f_m] \d q_m) =
    [q_1] \d f_1 + \cdots +  [q_m]\d f_m.
  \]
  Hence, $U$ is contained in the
  $k(P)$-linear
  span of the $\d f_i$.
\end{proof}
And now the technical result where we need characteristic zero.  For a
vector $v = \alpha_1\d x_1 + \cdots + \alpha_n \d x_n \in V$, we
define its support to be the set of $\d x_i$
with $\alpha_i \neq 0$.  We also recall that, for $f\in S$, the support
of $f$ is the set of $x_i$ so that
the variable $x_i$ appears in some monomial of $f$.
\begin{lemma}\label{lem: diff support}
  Let $P = \langle f_1, \ldots, f_m\rangle \subseteq S$
  a prime ideal and $k$ have characteristic zero.  Define $U$ to be the $k(P)$-linear span of $\{\d f_1, \ldots, \d f_m\}$.
  If $v$ is in $U$, 
  then there is a polynomial $g\in P$ so that
  $x_i$ is in the support of $g$ if and only if
  $\d x_i$ is in the support of $v$.
\end{lemma}
\begin{proof}
  Let
  \[
    v = \alpha_1 \d f_{1} + \cdots + \alpha_m \d f_{m}
  \]
  be a given vector in $U$, with $\alpha_i = [q_i]/[p_i]$ and $[p_i]\neq 0$.
  Because $P$ is
  prime, $p = p_1\cdots p_m$ is not in $P$.  Hence $[p]v$ is nonzero
  and has the same
  support as $v$, so we may assume that
  \[
    v = [q_1] \d f_{1} + \cdots + [q_m] \d f_{m}.
  \]
  Let $g = q_1f_{1} + \cdots + q_mf_m$.  Because $f_i\in P$, \eqref{eq: prodrule}
  implies that 
  $v = \d g$.  Now suppose that $\d x_i$ is not
  in the support of $v$.  Then $(\partial/\partial x_i)g = 0$, which
  implies, because
  $k$ has characteristic zero,
  that $x_i$ is not in the support of $g$.  On the other hand, if
  $\d x_i$ is in the support of $v$, we must have
  $(\partial/\partial x_i)g \neq  0$ (true in any characteristic)
  which implies that
  $x_i$ is in the support of $g$.
  Hence, the two supports are equal.
\end{proof}

\subsection{Specialization}  \label{sec: special} In the proof of Theorem \ref{thm: K/k-linear} (2),
we use an intermediate result about linear matroids that is
interesting in its own right.
Let us set up terminology for some specific linear matroids.  For an
$r\times n$ matrix
$A$ in a field, denote by $M(A)$ the linear matroid with
ground set $\{1,\ldots, n\}$ and independent sets $I$ such that the columns of
$A$ indexed by $I$ are linearly independent. 
If $B\subseteq \{1,\ldots, n\}$ define
$A\llbracket B\rrbracket$ to be  the submatrix of $A$ with columns indexed
by $B$.  The last
concept we need is
the specialization of a matrix $A\in k(P)^{r\times n}$ to $k$; if
$z\in V(P)$, we denote
by $A_z\in k^{r\times n}$ the matrix we get by evaluating 
any coset representatives of
entries
of $A$ at $z$.
\begin{theorem}\label{thm: specialization}
  Let $k$ be an algebraically closed field of characteristic zero, $S
  = k[x_1, \ldots, x_n]$, and $P \subseteq S$ be a prime ideal.
  Suppose that $A$ is a matrix with entries in $k(P)$.  There exists
  a nonempty set $U \subseteq V(P)$ such that for every $u \in U,$
  $M(A) = M(A_u)$.
\end{theorem}
\begin{proof}
  Because $k$ is algebraically closed, $X = V(P)$ is nonempty by Lemma \ref{lem: noth} (a).
  By Lemma \ref{lem: noth} (b), if $f\in S-P$ is nonzero, $V(P)\cap
  V(f)$ (which may be empty)
  is a proper  algebraic subset of $V(P)$.

  We describe the set $U$ by defining two algebraic subsets that we
  remove from $X$
  In order to evaluate $A$ at $z \in k^n$, we need to make sure that
  every element of $A$ is defined at $z$.
  Every nonzero entry $A_{ij}$ of $A$ has a representative
  $[f_{ij}]/[g_{ij}]$, where $f_{ij}, g_{ij}\in S-P$ are nonzero.  Let
  $D$ be the union of $V(g_{ij})$ over all entries $A_{ij}$. If $z
  \in V(P)-D$, every element of $A_z$ is defined.

  We may need to remove more of $V(P)$ to ensure that the specialized matrix
  $A_z$ has $M(A_z)$ equal to $M(A)$.
  Let $r = \rank M(A)$.  For each basis $B$ of
  $M(A)$ the columns of $A\llbracket B\rrbracket$  are linearly
  independent over $k(P)$.
  Hence, there exists an $r\times r$ minor $m_B$ that is nonzero in
  $k(P)$.  We may write $m_B = [f_B]/[g_B]$ where $f_B,g_B \in S-P$
  are nonzero. Let $W$ be the union of all sets $V(f_B)\cup V(g_B)$
  over all bases $B$ of $M(A)$.
  Define $U = V(P)- (D \cup W).$ By construction, $V(P)\cap (D\cup
  W)$ is the union of finitely many proper algebraic subsets of $V(P)$.
  By Lemma \ref{lem: noth} (b),     $V(P) \cap (D \cup W)\subsetneq
  V(P)$. Therefore, $U$ is nonempty.

  Let $u  \in U$.  Suppose that $B$ is a basis of $M(A)$.  There
  is a nonzero $r\times r$ minor of $A\llbracket B\rrbracket$, which
  is a rational
  function $m_B \in k(P)$ certifying that $A\llbracket B\rrbracket$ has linearly
  independent columns.
  Since $u \notin W$, if $B$ is a basis of $M(A)$ then
  $m_B(u)$, which is equal to the corresponding minor of $A_u$,
  is nonzero. Hence, the corresponding columns of $A_u$ are a basis in
  $M(A_u)$.

  For the other direction, suppose that $B$ is a basis of $M(A_u)$.
  Then there must be an $r\times r$ minor of  $A_u(B)$ that is nonzero
  in $k$.  Since this value is obtained by evaluating
  the corresponding minor of $A\llbracket B\rrbracket$ at $u$, it
  must be that this minor
  is nonzero in $k(P)$ and the columns of $A\llbracket B\rrbracket$ are linearly
  independent.  Since we have shown that $M(A)$ and
  $M(A_u)$ have the same bases,
  we conclude that the matroids are equal.
\end{proof}

\subsection{Proof of Theorem \ref{thm: K/k-linear}}\label{sec: proof}
\paragraph{(1)}
Let $S = k[x_1, \ldots, x_n]$ and $P = \langle f_1, \ldots,
f_m\rangle$ be prime.
Suppose that $k$ has characteristic zero.  Let $V$
be the $k(P)$-vector space spanned by $\d x_1, \ldots, \d x_n$
and let $U$ be the subspace of $V$ spanned by $\d f_1, \ldots, \d f_m$.
Define $\Omega_{k(P)/k}$ to be $V/U$.  We claim the linear matroid
defined by the images of the $\d x_i$ in $\Omega_{k(P)/k}$ is equal to
$M(P)$.

To this end, we first suppose that $D = \{x_{i_1}, \ldots, x_{i_s}\}$
is a dependent set in $M(P)$.
Then there is an $f\in P$ with its support supported contained in $D$.
Hence the support of $\d f$ is contained in $D' = \{\d x_{i_1},
\ldots, \d x_{i_s}\}$.
By Lemma \ref{lem: diff spans},
$\d f\in U$ and, hence, the set $D'$ is linearly dependent
in $\Omega_{k(P)/k}$.  For the other direction, we suppose that
$D = \{\d x_{i_1}, \ldots, \d x_{i_s}\}$ is dependent in the matroid of the $\d
x_i$ in $\Omega_{k(P)/k}$.
This means that there are nonzero coefficients $\alpha_1, \ldots,
\alpha_s\in k(P)$ such that
\[
  v = \alpha_1\d x_{i_1} + \cdots + \alpha_s \d x_{i_1} \in U.
\]
By Lemma \ref{lem: diff support}, there is a $g\in P$ with support
corresponding to that of $v$.
Hence $\{x_{i_1}, \ldots, x_{i_s}\}$
is dependent in $M(P)$.  Since we have a bijection between dependent sets,
the map $\varphi : \{x_1, \ldots, x_n\}\to \Omega_{k(P)/k}$ given by
$\varphi(x_i)\mapsto \d x_i$  shows that $M(P)$ is $k(P)$-representable.\\

\paragraph{(2)} Now suppose further that $k$ is algebraically closed.  Let
\[
  J =
  \begin{pmatrix}
    \frac{\partial}{\partial x_j}f_i
  \end{pmatrix}\in k(P)^{m\times n}
\]
be the Jacobian matrix of the given generators of $P$, and suppose that the
rank of $J$ is $n - r$.
Let $A\in k(P)^{r\times n}$ be a matrix such that the nullspace of
$A$ is equal to the
column space of $J^T$.  By the first isomorphism theorem, the column
space of $A$ is
linearly isomorphic to $\Omega_{k(P)/k}$, and, under the isomorphism the $i$th
column of $A$ gets mapped to the image of $\d x_i$ in $\Omega_{k(P)/k}$.  Hence,
the matroid $M(A)$ is isomorphic to the one constructed in part (1).  Since $k$ is
algebraically closed and of characteristic zero, Theorem \ref{thm: specialization}
implies that
there is a $z\in V(P)$ such that $M(A_z)$ and $M(A)$  are isomorphic.  Hence
$M(A_z)$ is a $k$-representation of the $k$-algebraic matroid $M(P)$.
\hfill $\qed$

We now present a worked example  
illustrating the construction in the proof of 
Theorem~\ref{thm: K/k-linear}.
\begin{example}
Let $S$ be the polynomial ring over $\mathbb{C}$ with variables given by the 
entries of the matrix 
\[
    X = \begin{pmatrix}
        x_{11} & x_{12} & x_{13} \\
        x_{21} & x_{22} & x_{23}
    \end{pmatrix}.
\]
The prime ideal $P$ cutting out the matrices of rank at most one is
\[
    P = \langle \begin{array}{ccc} x_{11}x_{22} - x_{12}x_{21}, &x_{11}x_{23} - x_{13}x_{21}, & x_{12}x_{23} - x_{13}x_{22} \end{array}\rangle \subseteq S.
\]
To find a $\mathbb{C}(P)$ representation of $M(P)$, we first compute the Jacobian of the generators
\[
    J = \left(
\begin{array}{cccccc}
 x_{22} & -x_{21} & 0 & -x_{12} & x_{11} & 0 \\
 x_{23} & 0 & -x_{21} & -x_{13} & 0 & x_{11} \\
 0 & \phantom{-}x_{23} & -x_{22} & 0 & -x_{13} & x_{12} \\
\end{array}
\right).
\]
Row reduction over $\mathbb{C}(P)$ yields the matrix 
\[
\left(
\begin{array}{cccccc}
 1 & 0 & -x_{21}/x_{23} & -x_{13}/x_{23} & 0 & x_{11}/x_{23} \\
 0 & 1 & -x_{22}/x_{23} & 0 & -x_{13}/x_{23} & x_{12}/x_{23} \\
 0 & 0 & 0 & 0 & 0 & 0 \\
\end{array}
\right),
\]
\noindent from which we read off the $\mathbb{C}(P)$ representation given by the columns of 
\[
\begin{blockarray}{cccccc}
 \phantom{-}\mathrm{d}x_{11} &  \phantom{-}\mathrm{d}x_{12} & \mathrm{d}x_{13} &
\mathrm{d}x_{21} & \mathrm{d}x_{22} & \mathrm{d}x_{23} \\[3mm]
\begin{block}{(cccccc)}
 \phantom{-}x_{21}/x_{23} & \phantom{-}x_{22}/x_{23} & 1 & 0 & 0 & 0 \\
 \phantom{-}x_{13}/x_{23} & 0 & 0 & 1 & 0 & 0 \\
 0 & \phantom{-}x_{13}/x_{23} & 0 & 0 & 1 & 0 \\
 -x_{11}/x_{23} & -x_{12}/x_{23} & 0 & 0 & 0 & 1 \\
\end{block}
\end{blockarray}\:\: .
\]

\noindent As a sanity check, from the zero pattern, we see immediately that 
$\{\d x_{11}, \d x_{13}, \d x_{21}, \d x_{23}\}$ and 
$\{\d x_{12}, \d x_{13}, \d x_{22}, \d x_{23}\}$ are dependent sets, 
which correspond to the generators $x_{11}x_{23} - x_{13}x_{21}$ and 
$x_{12}x_{23} - x_{13}x_{22}$ of $P$, respectively.  A direct computation also gives that 
$\{\d x_{11}, \d x_{12}, \d x_{21}, \d x_{22}\}$ is a dependent set.
If we then specialize $X$ to 
\[
\begin{pmatrix}
    1 & 2 & 3 \\
    1 & 2 & 3
\end{pmatrix},
\]
we get the matrix  
\[
\left(
\begin{array}{cccccc}
 \frac{1}{3} & \frac{2}{3} & 1 & 0 & 0 & 0 \\
 1 & 0 & 0 & 1 & 0 & 0 \\
 0 & 1 & 0 & 0 & 1 & 0 \\
 -\frac{1}{3} & -\frac{2}{3} & 0 & 0 & 0 & 1 \\
\end{array}
\right),
\]
which determines the same linear matroid.
\end{example}

We end this section with a remark on computational 
issues.
\begin{remark}
Theorem \ref{thm: K/k-linear}(1) gives us a  representation of our algebraic matroid $M(P)$ as a linear matroid over $k(P)$.
Using this, we may test for independence in $M(P)$ by applying Gaussian elimination to the corresponding vectors in a $k(P)$-vector 
space.  One may ask why this gives us any additional efficiency, since field operations in  
$k(P)$ require a Gr\"{o}bner basis, just as naive elimination in $P$ does! 
However, checking whether $[g] \equiv 0 \mod P$ can be done using a Gröbner basis for $P$ constructed 
using any term order, while constructing elimination ideals depends on a potentially 
much worse elimination ordering.

Passing to the $k$-representations given by Theorem \ref{thm: K/k-linear}(2)
requires finding a general point on $V(P)$, which is a computationally 
difficult problem.  Theorem \ref{thm: param}, in the next section, 
discusses a special case in which the $k$-representation can be 
found in polynomial time. 
\end{remark}

%% file: parameterized_varieties.tex
\section{What if the Field is Not Algebraically Closed?}\label{sec: char zero}

If $k$ has characteristic zero but is not algebraically closed we may
need to enlarge the field to be able to produce a linear
representation from an algebraic one.  Indeed, as a consequence of
Mnëv--Sturmfels universality \cites{M85,M88,M91,S87} (see also
\cite{RG95,AP17}) there exists a matroid $M$ that is $\mathbb{Q}$-algebraic
but not $\mathbb{Q}$-representable. In Example~\ref{ex: perles} we reproduce an example derived from Perles's configuration (see \cite[Example 12.3]{Pak10}) illustrating this phenomenon.  The construction of 
Theorem~\ref{thm: K/k-linear} fails in this example because the variety corresponding to the given ideal has no smooth rational points.

\begin{example}
\label{ex: perles}

  Let $K =\mathbb{Q}[t]/\langle t^2+t-1\rangle.$
  The matroid $M$ is the rank $3$ linear matroid determined by the
  columns of the matrix:
  \[
    A =
    \begin{pmatrix}1&1&1&0&0&0&1&1&1+t\\
      t&0&-1&1&0&1&1+t&0&1\\
      0&0&0&0&1&1&1&1&1
    \end{pmatrix}.
  \]
  By construction, $M$ has a $K$-representation.
  Pak \cite[Example 12.3]{Pak10}
  shows that this matroid is not $\mathbb{Q}$-representable.
  To see that $M$ is $\mathbb{Q}$-algebraic, we construct a prime
  ideal $P$ in $\mathbb{Q}[x_1, \ldots, x_9]$ such that $M(P)$ is
  isomorphic to $M$.
  Define an auxiliary ideal $P'$ of $\mathbb{Q}[x_1, \ldots, x_9,t]$ by
  \begin{align*}
    P' = \langle &-x_{2}+x_{3}+x_{4},\,-x_{2}+x_{3}-x_{5}+x_{6},\,-x
    _{2}-x_{5}+x_{8},\\
    & -tx_{2}+t_{3}+x_{1}-x_{2},\,-tx_{2}+tx_{3}-2x_{2}+x
    _{3}-x_{5}+x_{7},\\
    &-tx_{2}-2x_{2}+x_{3}-x_{5}+x_{9},\,t^{2}+t-1\rangle.
  \end{align*}
  The linear forms in the list of generators of $P'$
  generate the the relations on the columns of $A$.  Eliminating $t$
  from $P'$ gives a prime ideal in $\mathbb{Q}[x_1, \ldots, x_9],$  which is
  \[
    \begin{array}{l}
      \langle x_4+x_5-x_6,\hspace{.6cm} x_3+x_6-x_8, \hspace{.6cm}
      x_2+x_5-x_8,\hspace{.6cm}
      x_1+x_5-x_9, \\
      x_5x_7-x_7^2-x_6x_8+x_8^2-x_5x_9+x_6x_9+x_7x_9-x_8x_9, \\
      x_6^2-x_6x_7-x_7^2-2x_6x_8+x_7x_8+x_8^2+x_6x_9+2x_7x_9-x_8x_9-x_9^2,\\
      x_5x_6-x_6x_7-x_5x_8-x_6x_8+x_8^2+x_6x_9+x_7x_9-x_9^2, \\
      x_5^2-x_7^2-x_5x_8-2x_6x_8+x_7x_8+x_8^2-x_5x_9+2x_6x_9+x_7x_9-x_9^2
      \rangle.
  \end{array}\]
  We define $P$ to be this ideal.  The algebraic matroid $M(P)$ is isomorphic to
  $M$ because the points of $V(P)$ are parameterized by the row space of $A$.
  The (projective) variety $V(P$)  has only one rational point, which
  is singular. It
  follows that the set of points in $V(P)$ satisfying the conclusion of
  Theorem~\ref{thm: specialization} is empty.
\end{example}

Let us revisit Example \ref{ex: diff} in light of everything we have
seen.  Even though $\mathbb{R}$ is {\em not} algebraically closed, the
Jacobian construction worked, and, when we specialized $A$,
we didn't just use a real point, we used a {\em rational} one.
\begin{example}\label{ex: rational}
  Let $P = \langle x_{00}x_{11} - x_{01}x_{10}\rangle$ be the ideal
  from Example \ref{ex: prob}.  The
  discussion of the connection to the indepence model from algebraic 
  statistics in Example \ref{ex: prob} shows that $V(P)$ is the
  Zariski closure of the image of the map
  \[
    (p_0,p_1,q_0,q_1)\mapsto
    \begin{pmatrix}
      q_0 \\ q_1
    \end{pmatrix}
    \begin{pmatrix}
      p_0 & p_1
    \end{pmatrix},
  \]
  which is defined by polynomials with rational coefficients.  Since
  the rational
  points are dense in $\mathbb{R}^4$, and the parameterizing map sends rational
  points to rational points, these points are dense in $V(P)$ as
  well.  Because the set of points where the Jacobian construction fails is 
  nowhere dense, we can find a rational point that specializes to the right 
  matroid.
\end{example}
This general setting is common in applied algebraic geometry.  
In fact, several of the applications listed in the introduction have the
following setup. We fix an affine parameter space $k^m$ and a polynomial map
\[
  f : k^m \to k^n\qquad f(z) = (f_1(z), \ldots, f_n(z)),
\]
where each of the $f_i$ is in $S = k[x_1, \ldots, x_m]$.
The ideal $P$ is then the vanishing ideal $I(f(k^m))$ of the image of
$f$.  In this setup, a hypothesis on $f$, which is often satisfied in
practice, lets us substantially relax the hypothesis on $k$.  If
$k'$ is a subfield of $k$, we say that $f : k^m\to k^n$ is
{\em defined over $k'$} if all the coefficients of the
components of $f$ lie in $k'$.
\begin{theorem}[{cf \cite[Sec. 5.2]{GHT}}]\label{thm: param}
  Let $k$ be a field of characteristic zero, and denote by 
  $\overline{k}$ the algebraic closure of $k$.  Let $k'\subseteq k$  be a subfield of $k$.   Suppose that  $P\subseteq S$ is the vanishing ideal of 
  the image of a polynomial map   $f : k^m\to k^n$ that is defined over $k'$.  
  Then the algebraic matroid $M(P)$ is $k'$-representable.
\end{theorem}
\begin{proof}[Proof sketch]
  The map $f$ extends to a map $\overline{k}^m\to \overline{k}^n$ in a
  natural way%
  \footnote{Because $(\overline{k})^n\cong k^n\otimes_k \overline{k}$.}
  that preserves the property that it sends $k'$-points to $k'$-points.
  In particular, the $k'$-points of the image of $f$ are Zariski dense.
  Because we have passed to $\overline{k}$, Theorem \ref{thm: K/k-linear}
  now applies.  The set $U$ used in the proof of
  Theorem \ref{thm: K/k-linear} (2) is open, and so, by the previous
  discussion, it contains a $k'$-point, which is also in the image of the
  original map $f$.
\end{proof}
Informally, Theorem \ref{thm: param} says that if a variety is parameterized, the image of $f$ contains ``enough'' $k'$-points 
so that they must appear in the set $U$.  
The most commonly occurring 
case, and the one to keep in mind as an example, is when 
$k' =\mathbb{Q}$, $k = \mathbb{R}$, and
$\overline{k} = \mathbb{C}$, as in Example \ref{ex: rational}.
In this case, Theorem \ref{thm: param} gives an efficient 
algorithm to find a $\mathbb{Q}$-representation of $M(P)$.
We pick a uniformly random element $z$ of a cube $Q$ in 
$\mathbb{Z}^m\subseteq \mathbb{Q}^m$ and then specialize the 
Jacobian $J$ from the proof of Theorem \ref{thm: K/k-linear},
at $f(z)$. It is shown in \cite{GHT}, using the Schwartz--Zippel 
Lemma \cite{Schwartz1980}, how to pick the cube so that, with high probability, 
this process results in a $\mathbb{Q}$-representation of $M(P)$.  
Because all the computations take place in $\mathbb{Q}$, 
the process is in randomized polynomial time.  An alternative computational 
approach, which is familiar from rigidity theory, is to use instead 
$(\d f)^t_z$ as the $\mathbb{Q}$-representation.  The connection between 
the two is that, if $g\in P$, then the chain rule gives 
$(\d g)_{f(z)}(\d f)_z = 0$, so $(\d f)^t$ provides the matrix $A$
from the proof of Theorem \ref{thm: K/k-linear} directly.

%% file: history.tex
\section{The Wild World of Characteristic $p$}\label{sec: lindstrom}
In this section, we show that the construction in the proof of
Theorem~\ref{thm: K/k-linear} (1) may fail if $k$
has characteristic $p$ using an example of Lindstr\"om. 
We also briefly discuss the idea of a matroid flock introduced in \cite{BDP}, which associates a collection of linear matroids to an algebraic matroid over a field of positive characteristic.

\subsection{The Non-Pappus Matroid is algebraic but not linear}
In this section we describe work of Lindström \cite{L83,L86}, showing that, for each prime $p$, the so-called non-Pappus matroid is $\operatorname{GF}(p^2)$-algebraic but not linear and illustrate how the Jacobian construction fails for this construction.  

The non-Pappus matroid derives its name from its connection to Pappus's Theorem in projective geometry which says that if 
$v_1, v_2, v_3$ and $v_4, v_5, v_6$ are collinear points in the projective plane, then 
the points $v_7, v_8, v_9$ defined by the intersections of the 
lines $v_7 = v_1v_5\cap v_2v_4$, $v_8 = v_1v_6\cap v_3v_4$, $v_9 = v_2v_6\cap v_3v_9$
must be collinear, as illustrated in Figure~\ref{fig: non-pappus}.  The non-Pappus matroid is the rank 3 matroid with ground set $\{1,\ldots, 9\}$ whose dependent 3-element sets correspond to the triples of points on the black lines in Figure~\ref{fig: non-pappus}.  So, $\{1,2,3\}, \{1,5,7\}, \ldots$ are dependent, but crucially, $\{7,8,9\}$ is an independent set.  
This implies that the non-Pappus matroid is not $k$-representable for any field $k$.
\begin{figure}[t]
  \begin{tikzpicture}[scale=3,
      node style/.style={draw=black, thick, circle, fill=white,
      minimum size=4pt, inner sep=0pt},
    font=\small]

    \coordinate (1) at (0,0);
    \coordinate (2) at ($(1)+(1, -0.087)$);  
    \coordinate (3) at ($(2)+(0.8, -0.07)$);  

    \coordinate (4) at (0,-1);
    \coordinate (5) at ($(4)+(1.2, -0.212)$);  
    \coordinate (6) at ($(5)+(0.7, -0.12)$);   

    \path[name path=L15] (1) -- (5);
    \path[name path=L24] (2) -- (4);
    \path[name intersections={of=L15 and L24, by=7pt}];

    \path[name path=L16] (1) -- (6);
    \path[name path=L34] (3) -- (4);
    \path[name intersections={of=L16 and L34, by=8pt}];

    \path[name path=L26] (2) -- (6);
    \path[name path=L35] (3) -- (5);
    \path[name intersections={of=L26 and L35, by=9pt}];

    \draw (1) -- (5);
    \draw (1) -- (6);
    \draw (2) -- (4);
    \draw (2) -- (6);
    \draw (3) -- (4);
    \draw (3) -- (5);
    \draw (4) -- (6);
    \draw (1) -- (3);
    \draw[red, dotted] ($(7pt)!-3mm!(9pt)$)-- ($(9pt)!-3mm!(7pt)$);

    \foreach \v in {1,2,3,4,5,6,7pt,8pt,9pt}
    \node[node style] at (\v) {};

    \node[draw=none] at ($(1)+(0,0.08)$) {$v_1$};
    \node[draw=none] at ($(2)+(0,0.08)$) {$v_2$};
    \node[draw=none] at ($(3)+(0,0.08)$) {$v_3$};
    \node[draw=none] at ($(4)+(0,-0.08)$) {$v_4$};
    \node[draw=none] at ($(5)+(0,-0.08)$) {$v_5$};
    \node[draw=none] at ($(6)+(0,-0.08)$) {$v_6$};
    \node[draw=none] at ($(7pt)+(0,-0.08)$) {$v_7$};
    \node[draw=none] at ($(8pt)+(0,-0.08)$) {$v_8$};
    \node[draw=none] at ($(9pt)+(0,-0.08)$) {$v_9$};

  \end{tikzpicture}
  \caption{Pappus's Theorem and the non-Pappus matroid}
  \label{fig: non-pappus}
\end{figure}

Lindstr\"om's algebraic representation of the non-Pappus matroid is derived from a ``geometric'' representation of the matroid as a configuration of nine points in $R^3$, where $R$ is a noncommutative ring of characteristic $p.$  To gain intuition for why such a representation is possible, we note that although Pappus's Theorem is regarded as a geometric result, it can also be interpreted as a result about linear algebra.  In fact, Pappus's Theorem can be proved using determinants as in \cite[Example 3.4.3]{S}. However, the determinantal proof relies on the fact that arithmetic over a field $k$ is commutative, so it breaks down over a non-commutative ring. 

Lindström is able to translate his geometric representation into an algebraic one.  For example, when $p=5$, his construction produces the ideal
\begin{equation*}\label{eq: non-pappus}
  \begin{aligned}
    P = \langle
    & x_{4}-x_{5}+(\alpha-1)x_{6}, \\
    & x_3-x_5+x_9,\\
    & x_{2}-\alpha x_{6}+x_{9}, \\
    & x_{1}+2\alpha    x_{5}-2\alpha x_{7}, \\
    &
    x_{9}^{5}+(-2\alpha-2)x_{5}+(2\alpha-1)x_{6}+(2\alpha+1)x_{7}+(-\alpha+1)x_{8},
    \\
    & x_{6}^{5}+2\alpha x_{5}-x_{6}-2\alpha x_{7}+x_{8} \rangle,
  \end{aligned}
\end{equation*}
where $\alpha \in \operatorname{GF}(p^2)$ is any scalar such that $\alpha^5\neq \alpha$.

However, the Jacobian construction used in the proof of Theorem \ref{thm: K/k-linear} (1) 
fails on the example, as it must, because the non-Pappus matroid is not linearly representable. 
Because $\Char k = 5$, we have $\d(x^5) = 0$, and so the
differentials of the last two generators have the same support.
Specifically, they are:
\begin{equation*}\label{eq: support change}
  (-2\alpha-2)\d x_{5}+(2\alpha-1)\d x_{6}+(2\alpha+1)\d x_{7
  }+(-\alpha+1)\d x_{8}
\end{equation*}
and
\[2\alpha\,\d x_{5}-\d x_{6}-2\alpha\,\d x_{7}+\d x_{8}.
\]
Canceling $\d x_5$ we get a nontrivial linear relation on $\d x_6,
\d x_7,$ and $\d x_8.$
However, $\{x_6,x_7,x_8\}$ is a basis in the algebraic matroid of $P$,
so the linear matroid on the differentials is not isomorphic to the
algebraic matroid,
even in $k(P)$.

\begin{remark}
  From the results in this section, for each prime $p$, there exists
  a field $k$ with
  characteristic $p$ such
  that the non-Pappus matroid is $k$-algebraic.
  At the other extreme, if, for some field $k$, the Fano matroid is
  $k$-algebraic,
  the characteristic of $k$ is $2$.
  In general, for a
  fixed matroid, one
  can ask for the set of all characteristics for which it has an
  algebraic representation.
  This question was first explored in \cite{L85,L86}; only very recently, it was almost completely settled by \cite{CV} using tools from \cite{EH}.
\end{remark}

\subsection{Frobenius Flocks}
As Lindstr\"om's algebraic representation of the non-Pappus matroid
shows, the Jacobian
construction from the proof of
Theorem~\ref{thm: K/k-linear} (1) may not produce a
$k(P)$-representation of $M(P)$ if $\Char k = p$.  However, following ideas of Lindstr\"om \cite{L85}, Bollen, Draisma, and Pendavingh \cite{BDP} describe a beautiful way to exploit the Jacobian construction using the Frobenius map, producing a ``flock'' of linear matroids associated to $M(P)$.  We give intuition for their construction in this section.

We begin with a toy example that illustrates Lindstr\"om's idea of applying the Frobenius map to the coordinates of the ambient space and then using the Jacobian construction.

\begin{example}
  Let $k = \mathbb{Z}/3\mathbb{Z}$, and let $P = \langle x^3 + x^{6}y - z\rangle
  \subseteq k[x,y,z]$.  Because $P$ has a single generator supported on all 
  the variables, 
  $M(P)$ is the uniform matroid of rank $2$ on three
  elements, which is
  $k(P)$-representable (even $k$-representable).  However, the
  Jacobian construction
  will not show us this, since
  \[
    \d( x^3 + x^{6}y - z) =
    x^{6}\d y - \d z,
  \]
  so $y$ and $z$ are dependent in the linear matroid on
  differentials.

  Consider the map that sends $(a,b,c) \mapsto (a^3,b,c)$ for
  $(a,b,c) \in k^3$.  To see the effect of this map on ideals, define
  $\widehat{P} = \langle x^3 + x^{6}y - z, t - x^3\rangle\subseteq k[x,y,z,t]$.
  Then the image of $V(P)$ has ideal
  \[
    P' = \widehat{P}\cap k[y,z,t] = \langle t + t^2y -z \rangle.
  \]
  We compute
  \[
    \d (t + t^2y -z) = (1 + 2yt)\d t + t^2\d y - \d z,
  \]
  which gives us a uniform matroid (but not the differential matroid
  we got from $P$).
\end{example}

In the example, we rescued the Jacobian construction
by sending $P$ to  $P'$ in a way that preserves the algebraic matroid but forces
the linear matroid on the coordinate differentials to change.  
In this case, the generator of $P'$ has no $p$-th powers, 
and so the differential matroid is isomorphic to $M(P)$.

Since we know that the non-Pappus matroid is not $k$-representable
for any field $k$, this approach,
suggested by Lindström \cite{L85}, is inherently limited.  Nonetheless, by considering all of the ways of 
applying Frobenius powers to the coordinates of our ambient space, we collect an assortment of imperfect 
linearizations, which we can think of as blurry snapshots of the original 
algebraic matroid.  The set of linear matroids obtained in this way is a ``matroid flock.'' In \cite{BDP}, 
Bollen, Draisma, and Pendavingh prove a series
of striking results about matroid flocks, among them that the
matroid flock is an invariant of $M(P)$ and that it has a finite
description. Matroid flocks are yet another example of the
beautiful interplay between algebra, linear algebra, and
combinatorics, and how even a ``failed'' construction can give rise
to beautiful new mathematics.

%% file: acknowledgements.tex
We thank June Huh for bringing the error in \cite{RST} to our
attention;  David Zurieck-Brown for helpful conversations;  
Shin-ichi Tanigawa and Matt Larson for helpful feedback on an 
earlier draft; two anonymous referees for comments that contributed greatly to our revision; and ICERM for its hospitality during the 
semester program 
``Geometry of Materials, Packings and Rigid Frameworks''.

%% file: funding.tex
This work was supported by the National Science Foundation 
under grant number DMS-1929284 and UK Research and Innovation under 
the EPSRC Small Grants Scheme grant UKRI1112.

%% file: declarations.tex
The authors have no conflicts of interest to declare.

%% file: contributions.tex
ZR, JS and LST wrote and edited the manuscript.